\documentclass[11pt]{amsart}
\usepackage{verbatim,amssymb,amsmath,mathrsfs, graphicx,hyperref}
\usepackage{calc}
\usepackage{enumitem}
\usepackage[dvipsnames]{xcolor}
  \scrollmode

\newcommand{\R}{{\mathbb R}}
\newcommand{\N}{{\mathbb N}}

\newcommand{\C}{{\mathbb C}}
\newcommand{\EE}{{\mathbb E}}
\newcommand{\PP}{{\mathbb P}}

\DeclareMathOperator\supp{supp}

\newcommand{\eps}{\varepsilon}

\newcommand\norm[1]{\lVert#1\rVert}

\newcommand{\tphi}{\widetilde \varphi}

\newcommand{\bi}{\mathbf i}

\newcommand{\F}{\mathscr{F}}

\theoremstyle{plain}
\newtheorem{theorem}{Theorem}
\newtheorem{prop}{Proposition}
\newtheorem{lemma}{Lemma}
\newtheorem{cor}{Corollary}

\theoremstyle{definition}
\newtheorem{rem}{Remark}

\begin{document}

\title[Regularity properties of densities of SDEs]{Regularity properties of densities of SDEs using the Fourier analytic approach}

\author[Ellinger]
{Simon Ellinger}
\address{
Faculty of Computer Science and Mathematics\\
University of Passau\\
Innstrasse 33 \\
94032 Passau\\
Germany} \email{simon.ellinger@uni-passau.de}

\begin{abstract}	
	We show regularity properties of local densities of solutions of stochastic differential equations (SDEs) with the Fourier analytic approach. With this simple method, statements that were previously derived with approaches using Malliavin calculus or difference operators can be recovered and extended to include regularity properties with respect to the time variable. For example, we derive the Hölder continuity and joint continuity of local densities in the case of drift coefficients that are locally piecewise Hölder continuous. To this end, we derive fairly general bounds for the Fourier transform of the local density of a solution of the SDE when the drift is locally bounded and the diffusion is locally sufficiently regular. 
\end{abstract}
\maketitle

Consider a scalar autonomous stochastic differential equation (SDE)
\begin{equation}\label{sde0}
	\begin{aligned}
		dX_t & = \mu(X_t) \, dt +  \sigma(X_t) \, dW_t, \quad t\in [0,1],\\
		X_0 & = x_0
	\end{aligned}
\end{equation}
with deterministic initial value $x_0\in\R$, drift coefficient $\mu\colon\R\to\R$, diffusion coefficient $\sigma\colon \R \rightarrow \R$ and a one-dimensional driving
Brownian motion $W=(W_t)_{t\in[0,1]}$. Assume that the coefficients $\mu, \sigma$ are smooth enough such that there exists a strong solution of the SDE~\eqref{sde0}.

In this paper we study the local regularity of densities of solutions $X$ of the SDE~\eqref{sde0} for sufficiently regular coefficients $\mu, \sigma$. The existence of such densities
was proven in \cite{FP10} using a Fourier analytic approach. The main idea therein is to show that the characteristic function of $X_t$ is in $L^2(\R)$ which then implies the existence of a density. Thereby, the characteristic function of $X_t$ is bounded by approximating $X_t$ with  $Z_{t, \eps} = X_{t-\eps} + \sigma(X_{t-\eps})(W_t - W_{t-\eps})$ for $t \in (0,1], \eps \in (0,t)$. In many situations, in addition to the existence of densities, knowledge such as boundedness, continuity or Sobolev regularity would be advantageous. However, it was assumed, see e.g. \cite{debussche2013, Hayashi2014, Hayashi2012, Makhlouf2016}, that the simple method from \cite{FP10} is not suitable to derive such properties. As a result, the ideas of \cite{FP10} were further developed in two directions in order to obtain stronger statements regarding the regularity of the densities.

The first approach goes back to \cite{debussche2014} and relies mainly on the central idea of \cite{FP10} to approximate $X_t$ by $Z_{t, \eps}$. Instead of bounding the characteristic function of the random variable, one estimates difference operators appropriately. In this way, Besov regularity of global densities is obtained in \cite{debussche2014}. In \cite{rom18} this approach was further generalized and they provide bounds for the local $B_{1, \infty}^\beta$-norm of the local density, if $\mu$ is locally bounded, $\sigma$ is locally non-degenerate and $\alpha$-Hölder continuous where $\alpha \in (0,1)$, $\beta \in (0,\alpha)$, see Theorem 4.1 in \cite{rom18}. In this paper it was furthermore shown that the global density of $X_t$
has a Hölder continuous density if $\mu$ is bounded and $\sigma$ is elliptic and Lipschitz continuous.
A similar statement about the Hölder continuity of the density is also shown in \cite{Banos17} using a variation of the approach with the difference operator. For further results shown with the help of difference operators under global assumptions on the coefficients
see e.g. \cite{debussche2013, Friesen2021, rom16}.

The second extension is based on Malliavin calculus and was used among others in \cite{Hayashi2012}. There it was shown that the local density is Hölder continuous of order $\gamma \in (0,\alpha)$ if the function $\frac{\mu}{\sigma}$ is $\alpha$-Hölder continuous on an open interval $I$ and if $\sigma \in C^\infty_b(I)$ with $\inf_{x \in I} |\sigma(x)| > 0$ is the case. Furthermore, in \cite{deMarco11}, regularity properties of the local density of $X_t$ on an open interval $I$ are derived under the assumption $\mu, \sigma \in C^\infty_b(I)$.
Further results under global assumptions on $\mu, \sigma$, which are shown with the Malliavin calculus approach, can already be traced back to \cite{Bichteler1983}. Later results in this area can be found in \cite{Banos16, Bally2011, Hayashi2014, Hayashi2012, Pellat2024, Nakatsu2023, Nguyen2015, Olivera2019, Tahmasebi2014}.

In this paper we focus again on the original Fourier analytic approach of \cite{FP10}.  We will use this simple approach despite the above mentioned skeptical remarks in \cite{debussche2013, Hayashi2014, Hayashi2012, Makhlouf2016} to derive regularity properties for the local densities of the solutions of the SDE~\eqref{sde0} in a rather direct way. This is done with appropriate bounds on the Fourier transforms of localized measures, as can be seen in the following theorem.

\begin{theorem}\label{thm:localDensity}
	Let $\mu, \sigma\colon \R \rightarrow \R$ be measurable functions and assume that there exist $\xi \in \R$, $\delta \in (0,\infty)$ such that $\mu\vert_{[\xi - \delta, \xi+\delta]}$ is bounded, $\inf_{x \in B_\delta(\xi)} |\sigma(x)| \ge l_\sigma > 0$ for $l_\sigma \in (0, \infty)$ and $\sigma$ is Lipschitz continuous on $[\xi-\delta, \xi+\delta]$. Let $\sigma^\ast$ denote the constant continuation of $\sigma\vert_{[\xi - \delta, \xi+\delta]}$ given by
	\[
	\sigma^\ast = \sigma(\xi - \delta) 1_{(-\infty, \xi - \delta)} + \sigma 1_{[\xi-\delta, \xi + \delta]} + \sigma(\xi + \delta) 1_{(\xi + \delta, \infty)}.
	\]
	Let $X$ be a strong solution of the SDE~\eqref{sde0}, $\delta_0 \in (0, \delta)$ and $\varphi\colon \R \rightarrow \R$ be a differentiable function which has a Lipschitz continuous derivative $\varphi'$ and which satisfies $\supp(\varphi) \subset B_{\delta - \delta_0}(\xi)$. Then there exists a constant $c = c(\norm{\mu}_{L^\infty([\xi - \delta, \xi+\delta])}, \norm{\sigma^\ast}_{C^1(\R)}, l_\sigma,\norm{\varphi}_{C^2(\R)}, \delta_0) \in (0, \infty)$ such that for all $t \in (0,1]$ there is a function $p_t \in L^2(\R)$ with
	\begin{equation}\label{thm1_density}
		\varphi(x) \, \PP^{X_t}(dx) = p_t \bigl(\int_{\xi-\delta}^x \frac{1}{\sigma^\ast(z)} \, dz \bigr) \, \big/ \, |\sigma^\ast(x)| \, dx
	\end{equation}
	and for all $\eps \in (0, t)$ and $\lambda$-almost all $y \in \R$ it holds $| \F p_t (y) | \le \| \varphi \|_\infty$ as well as
	\begin{equation*}
		\begin{aligned}
			&|\F p_t (y)|  
			\\& \qquad \le c \Big[(1 +  \eps |y|)e^{-(\eps y^2)/2} + \eps \\
			&\qquad \qquad   + (1+ |y|) \EE \bigl [1_{\{ \forall s \in [t-\eps, t]:|X_s - \xi| \le \delta \}} \Big|\int_{t - \eps}^t \Big(\frac{\mu}{\sigma^\ast} - \frac{\delta_{\sigma^\ast}}{2}\Big)(X_s) - \Big(\frac{\mu}{\sigma^\ast} - \frac{\delta_{\sigma^\ast}}{2}\Big)(X_{t- \eps}) \, ds\Big|\bigr ] \Big]
		\end{aligned}
	\end{equation*}
	where $\delta_{\sigma^\ast}(x) = (\sigma^\ast)'(x)$ if $\sigma^\ast$ is differentiable in $x$ and $\delta_{\sigma^\ast}(x) = 0$ otherwise for $x \in \R$.
\end{theorem}

So the better you can estimate the object $\Big|\int_{t-\eps}^t(\frac{\mu}{\sigma^\ast} - \frac{\delta_{\sigma^\ast}}{2})(X_s) - (\frac{\mu}{\sigma^\ast} - \frac{\delta_{\sigma^\ast}}{2})(X_{t- \eps}) \, ds\Big|$ in the last term of Theorem~\ref{thm:localDensity}, the more regularity you get. It is worth pointing out that this was also recognized in the approach with the difference operator in the case when $\sigma = 1$, see Section 2.2 in~\cite{rom18}. An advantage of the pointwise bound for the Fourier transform of the density from Theorem~\ref{thm:localDensity} in comparison to the bounds for the Besov-norms in~\cite{rom18} is that this allows, for example, an application of the dominated convergence theorem which will later be used to obtain regularity of the densities with respect to the time.

It should be emphasized that the proof of Theorem~\ref{thm:localDensity} is based on elementary calculations and does not require knowledge about Besov theorey as in \cite{rom18} or Malliavin calculus as in \cite{deMarco11, Hayashi2012}. Furthermore, Theorem~\ref{thm:localDensity} can be used to derive statements for densities under global assumptions, which is illustrated in the following remark.

\begin{rem}
	Note that since the constant $c$ in Theorem~\ref{thm:localDensity} only depends on $\norm{\mu}_{L^\infty([\xi - \delta, \xi+\delta])}$, $\norm{\sigma^\ast}_{C^1(\R)}, l_\sigma,\norm{\varphi}_{C^2(\R)}, \delta_0$ one can derive regularity properties for densities of SDEs even in the global setting whenever $\sigma$ and $\frac{\mu}{\sigma^\ast} - \frac{\delta_{\sigma^\ast}}{2}$ are nice enough.  Consider therefore a sequence $(\varphi_k)_{k \in \N} \subset C^\infty(\R)$ such that $\varphi_k\vert_{[-k,k]} = 1$, $\supp(\varphi_k) \subset (-(k+1), k+1)$ and $\norm{\varphi_k}_{C^2(\R)} = \norm{\varphi_1}_{C^2(\R)}$ for all $k \in \N$ and then let $k$ go to infinity. Moreover, with slight modifications of the proof of Proposition~\ref{prop:fourierBound} and Theorem~\ref{thm:localDensity} one may only assume that $\mu$ satisfies the linear growth property instead of global boundedness in this case.
\end{rem}

In applications of Theorem~\ref{thm:localDensity} we will mostly use $\eps = \log^2(|y|) / y^2$ as in~\cite{FP10}. The resulting bound is stated in the next corollary.

\begin{cor}\label{cor1:fourierBound}
	Under the assumptions of Theorem~\ref{thm:localDensity} for all $t \in (0,1]$ there is a function \mbox{$p_t \in L^2(\R)$} which satisfies \eqref{thm1_density} and there exists a constant $c \in (0, \infty)$ such that for all $t \in (0,1]$ and $\lambda$-almost all $y \in \R$ with $|y| > 1$ and $\frac{\log^2(|y|)}{|y|^2} < t$ it holds $| \F p_t (y) | \le \| \varphi \|_\infty$ as well as
	\begin{equation*}
		\begin{aligned}
			|\F p_t (y)| & \le c \Big[ |y|^{-\log(|y|)/2} + \frac{\log^2(|y|)}{y^2} \\
			& \qquad +|y|\EE \bigl [1_{\{ \forall s \in [t-\eps_y, t]:|X_s - \xi| \le \delta \}} \Big|\int_{t - \eps_y}^t \Big(\frac{\mu}{\sigma^\ast} - \frac{\delta_{\sigma^\ast}}{2}\Big)(X_s) - \Big(\frac{\mu}{\sigma^\ast} - \frac{\delta_{\sigma^\ast}}{2}\Big)(X_{t- \eps_y})\, ds\Big|\bigr ] \Big]
		\end{aligned}
	\end{equation*} 
	where $\eps_y := \log^2(|y|) / y^2$.
\end{cor}

With the help of Theorem~\ref{thm:localDensity} and Corollary~\ref{cor1:fourierBound}, one may then prove continuity and boundedness of local densities. This is shown exemplarily for locally piecewise Hölder continuous coefficients. In this case the local densities are Hölder continuous, which can be seen from the following corollary. In particular, the statement is a generalization of Theorem 2 in~\cite{Hayashi2012} since $\sigma$ is allowed to be less regular. The next corollary covers also the special case of piecewise Lipschitz continuous coefficients. The existence of such SDEs has been investigated among others
in \cite{LS16, LS15b, MGY20,Soenmez2023}, where \cite{Soenmez2023} showed the existence of a density of the strong solution using the approach with the difference operator.

\begin{cor}\label{cor:HoelderReg}
	Under the assumptions of Theorem~\ref{thm:localDensity} and assuming that $\frac{\mu}{\sigma} - \frac{\sigma'}{2}$ is $\alpha$-Hölder continuous on $[\xi - \delta, \xi)$ and on $(\xi, \xi + \delta]$ for $\alpha \in (0,1]$ the measure $\varphi(x) \, \PP^{X_t}(dx)$ has a Lebesgue density $q_t \in L^1(\R)$ which is $\gamma$-Hölder continuous for all $\gamma \in (0, \alpha)$, where $t \in (0,1]$, and for all $t^\ast \in (0,1], \gamma \in (0, \alpha)$ it holds
	\begin{equation}\label{eqn_cor_localHoeld_1}
	\sup_{t \in [t^\ast, 1]} \norm{q_t}_{C^\gamma(\R)} < \infty.
	\end{equation}
	Moreover, the function
	\[
	(0,1] \times \R \rightarrow \R, \qquad (t,x) \mapsto q_t(x)
	\]
	is continuous.
\end{cor}

\begin{rem}
	It should be pointed out that with the approach using difference operators one can derive similar results as \eqref{eqn_cor_localHoeld_1} in Corollary~\ref{cor:HoelderReg}. This can be shown using ideas of Theorem 4.1 and Section 2.2 in~\cite{rom18} for the case $\sigma\vert_{[\xi - \delta, \xi+\delta]} = 1$ and afterwards one applies a Lamperti-type transform as in the proof of Theorem~\ref{thm:localDensity}. However, it does not seem to be possible to apply Theorem 4.1 in~\cite{rom18} directly, since the proof presented there is only valid for $\beta \in (0,1]$ and not for $\beta \in (0, \infty)$ as stated therein.
\end{rem}

The rest of this article is structured as follows. After an introduction of some notation in Section~\ref{sect:notation} we derive Theorem~\ref{thm:localDensity} in Section~\ref{sect:proofMainResult} where we use the Fourier analytic approach together with a Lamperti-type transformation. Subsequently, in Section~\ref{sect:proofCorollaries} we prove Corollary~\ref{cor:HoelderReg}.

\section{Notation}\label{sect:notation}

We denote the set of Borel-measurable sets in $\R$ by $\mathcal{B}(\R)$ and we write $\overline{A}$ for the closure of a set $A \in \mathcal{B}(\R)$. The open ball around $x \in \R$ with radius $\delta \in (0, \infty)$ is given by $B_\delta(x) = (x-\delta, x+\delta)$. Moreover, we denote as usual $s \land t = \min(s,t)$ for $s,t \in \R$.
For a function $\varphi\colon A \rightarrow \R$ we set $\norm{\varphi}_{L^\infty(A)} = \sup_{x \in A} |\varphi(x)|$ where $A \in \mathcal{B}(\R)$ and for simplification we set $\norm{\cdot}_\infty = \norm{\cdot}_{L^\infty(\R)}$. The support of a continuous function $\varphi\colon \R \rightarrow \R$ is denoted by $\supp(\varphi) = \overline{\{x \in \R \colon \varphi(x) \ne 0\}}$.
The function space of infinitely often differentiable functions on some open real interval $I$ is denoted by $C^\infty(I)$ and $C^\infty_b(I)$ is the space of all functions in $C^\infty(I)$ which are bounded and whose derivatives of any order are bounded as well. 
Furthermore, let
\[
[\varphi]_{C^\alpha(\R)} = \sup_{\substack{x,y \in \R \\ x \neq y}} \frac{|\varphi(x) - \varphi(y)|}{|x - y|^\alpha}
\]
denote the $\alpha$-Hölder constant of a function $\varphi\colon \R \rightarrow \R$ for $\alpha \in (0,1]$. Therewith, the space of $(k+\alpha)$-Hölder continuous functions for $k \in \N \cup \{0\}$ and $\alpha \in (0, 1]$ consists of all $k$-times differentiable functions $\varphi\colon \R \rightarrow \R$ with
\[
\norm{\varphi}_{C^{k+\alpha}(\R)} = \sum_{i=0}^{k} \norm{\varphi^{(i)}}_\infty + [\varphi^{(k) }]_{C^\alpha(\R)} < \infty.
\]
Moreover, for a function $\varphi\colon \R \rightarrow \R$ we set
\[
\delta_\varphi\colon \R \rightarrow \R, \quad x \mapsto 
\begin{cases}
	\varphi'(x), &\text{if $\varphi$ is differentiable in $x$}, \\
	0, & \text{otherwise}.
\end{cases}
\]

We denote by $L^p(\R)$ with $p \in [1, \infty)$ the space of measurable functions $\varphi\colon \R \rightarrow \C$ which satisfy
\[
\norm{\varphi}_{L^p(\R)} = \Bigl(\int_{\R} |\varphi(x)|^p\, dx\Bigr)^{1/p}<\infty.
\]
For two measurable functions $f,g\colon\R\to\C$ we write $f=g\text{ a.e.}$ whenever $\lambda(\{f\neq g\}) = 0$ and $f\le g\text{ a.e.}$ whenever  $\lambda(\{f> g\}) = 0$.
A version of the Fourier transform of a function $\varphi \in L^p(\R)$ for $p \in [1,2]$ is denoted by $\F \varphi$ and for $p = 1$ it is given by
\[
\F\varphi = \F[\varphi] = \int_\R e^{\bi z(\cdot)}\, f(z)\, dz\, \text{ a.e.}
\]
where $\bi$ denotes the imaginary unit.

Moreover, the Fourier transform of a finite signed measure $\mu$ on $(\R, \mathcal{B}(\R))$ is given by
\[
\F\mu = \F[\mu] = \int_\R e^{\bi z(\cdot)}\, \mu(dz).
\]

For a measure $\mu$ on $(\R, \mathcal{B}(\R))$ and a function $\varphi\colon \R \rightarrow \R$ which is integrable with respect to $\mu$ we write $\varphi(x) \, \mu(dx)$ for the signed measure given by
\[
\Big( \varphi(x) \, \mu(dx)\Big) (A) = \int_A \varphi(x) \, \mu(dx), \qquad A \in \mathcal{B}(\R).
\]

For a stopping time $\tau$ and an adapted process $X$ on some filtered probability space we write 
\[
X^\tau_\cdot = X_{\cdot \land \tau}
\]
for the stopped process.

\section{Proof of Theorem~\ref{thm:localDensity}}\label{sect:proofMainResult}

In the following let $(\Omega,\mathcal A, \PP)$ be a probability space and $W\colon [0,1] \times \Omega \rightarrow \R$ be a standard Brownian motion. We start with the proof of Theorem~\ref{thm:localDensity} for the special case $\sigma\vert_{[\xi-\delta, \xi+\delta]} = 1$ where we bound the Fourier transform of the localized measure with the help of an Euler-type approximation. Afterwards, we present a proof for Theorem~\ref{thm:localDensity} which essentially consists of two steps. First, we use Theorem~\ref{thm:localDensity} for the case that $\sigma$ is locally constant. Then the claim follows by applying a Lamperti-type transformation. We show in the following proposition a suitable bound for the Fourier transform of the localized measure for the case $\sigma\vert_{[\xi-\delta, \xi+\delta]} = 1$ where we use ideas from the proof of Theorem 2.1 in \cite{FP10}. 

 \begin{prop}\label{prop:fourierBound}
 	Let $\mu, \sigma\colon \R \rightarrow \R$ be measurable functions and assume that there exist $\xi \in \R$, $\delta \in (0,\infty)$ such that $\mu\vert_{[\xi - \delta, \xi+\delta]}$ is bounded and $\sigma \vert_{[\xi - \delta, \xi+\delta]} = 1$. Let $X$ be a strong solution of the SDE~\eqref{sde0}, $\delta_0 \in (0, \delta)$ and $\varphi\colon \R \rightarrow \R$ be a differentiable function which has a Lipschitz continuous derivative $\varphi'$ and which satisfies $\supp(\varphi) \subset B_{\delta - \delta_0}(\xi)$. Then there exists a constant $c = c(\norm{\mu}_{L^\infty([\xi - \delta, \xi+\delta])}, \norm{\varphi}_{C^2(\R)}, \delta_0)\in (0, \infty)$ such that for all $t \in (0,1]$ and $\varepsilon \in (0, t)$ it holds
 	\begin{equation*}
 		\begin{aligned}
 			&|\F[\varphi(x) \, \PP^{X_t}(dx)] (y)| \\
 			& \qquad  \le c \Big[(1 +  \eps |y|)e^{-(\eps y^2)/2} + \eps + (1+ |y|) \EE \bigl [1_{\{ \forall s \in [t-\eps, t]:|X_s - \xi| \le \delta \}} \Big|\int_{t - \eps}^t \mu(X_s) - \mu(X_{t- \eps}) \, ds\Big|\bigr ] \Big]
 		\end{aligned}
 	\end{equation*}
 	where $y \in \R$.
 \end{prop}

\begin{proof}
	For $t \in (0,1]$ and $\eps \in (0, t)$ we define inspired by Section 3 in \cite{debussche2013} and inequality (2.11) in~\cite{rom18} an Euler-type approximation for $X_t$ by
	\begin{equation*}
		Z_{t, \eps} := X_{t -\eps} + \eps \mu(X_{t-\eps}) + \sigma(X_{t-\eps})(W_t - W_{t-\eps}).
	\end{equation*}
	
	Let $t \in (0,1]$ and $\eps \in (0, t)$. Throughout this proof let $c_1,c_2,\dots\in (0,\infty)$ denote positive constants,
	which do only depend on $\norm{\mu}_{L^\infty([\xi - \delta, \xi+\delta])}, \norm{\varphi}_{C^2(\R)}, \delta_0$ and which neither depend on $t$ nor on $\eps$.	\pagebreak[2]
	\\
	\textit{Regularity of conditional characteristic function:}
	We start with the derivation of bounds for the conditional characteristic function of $Z_{t, \eps}$. Let $y \in \R$. Then it holds similar as in the proof of Theorem 2.1 in \cite{FP10} for $\PP^{X_{t - \eps}}$-almost all $x \in \R$
	\begin{equation} \label{prop1_1}
		\begin{aligned}
			|\EE[e^{\bi y Z_{t, \eps}} | X_{t-\eps} = x]| &= |e^{\bi yx} \cdot e^{\bi y\eps \mu(x)} \cdot \EE[e^{\bi y \sigma(x) (W_t - W_{t- \eps})}]| = e^{-y^2 \sigma^2(x) \eps / 2}.
		\end{aligned}
	\end{equation}
	Moreover, using \eqref{prop1_1} and the fact that for $N_\eps \sim N(0, \eps)$ it holds 
	\[
	\EE[N_\eps e^{\bi \hat{y} N_{\eps}}] = \bi \eps \hat{y} e^{- \hat{y}^2 \eps/2}, \qquad \hat{y} \in \R,
	\]
	we get for $\PP^{X_{t - \eps}}$-almost all $x \in \R$
	\begin{equation}\label{prop1_2}
		\begin{aligned}
			&|\EE[(Z_{t, \eps} - X_{t-\eps}) e^{\bi y Z_{t, \eps}} | X_{t-\eps} = x]|\\
			& \qquad \qquad = |\eps \mu(x)\EE[e^{\bi y Z_{t, \eps}} | X_{t-\eps} = x] + \sigma(x) \EE[(W_t - W_{t-\eps})e^{\bi y Z_{t, \eps}} | X_{t-\eps} = x]| \\
			&\qquad \qquad \le  \eps |\mu(x)| |\EE[e^{\bi y Z_{t, \eps}} | X_{t-\eps} = x]| + |\sigma(x) \cdot e^{\bi yx} \cdot e^{\bi y\eps \mu(x)} \cdot \EE[(W_t - W_{t-\eps}) e^{\bi y \sigma(x)(W_t - W_{t-\eps})}]| \\
			& \qquad \qquad \le \eps |\mu(x)|e^{-y^2 \sigma^2(x) \eps / 2} + \sigma^2(x) \eps |y| e^{- y^2 \sigma^2(x)\eps/2} \\
			& \qquad \qquad = \bigl (\eps |\mu(x)| + \eps |y| \sigma^2(x)\bigr) e^{- y^2 \sigma^2(x)\eps/2}.
		\end{aligned}
	\end{equation}
	\\
	\textit{Localization:}
	Now we localize the problem similar to the proof of Theorem 4.1 in \cite{rom18} and Section~5.1 in \cite{Hayashi2012}. Therefore, we use
	\[
	\tau_{t, \eps} := \inf\{ s \in [t-\eps, t] \colon X_s \notin (\xi - \delta, \xi + \delta)\}.
	\] 
	
	Note that since $\mu$ is bounded on $[\xi - \delta, \xi + \delta]$ and since $\sigma\vert_{[\xi - \delta, \xi + \delta]} = 1$, there exists for every $p \in [1, \infty)$ a constant $d_p = d_p(\norm{\mu}_{L^\infty([\xi - \delta, \xi+\delta])})\in (0, \infty)$ such that
	\begin{equation}\label{prop1_6}
		\begin{aligned}
			\EE[\sup_{s \in [t-\eps, t]} |X_s^{\tau_{t, \eps}} - X_{t-\eps}^{\tau_{t, \eps}}|^p] &= \EE \Big [\sup_{s \in [t-\eps, t]} \Big|\int_{t-\eps}^{s \land \tau_{t, \eps}} \mu(X_u) \, du + \int_{t-\eps}^{s \land \tau_{t, \eps}} \sigma(X_u) \, dW_u\Big|^p\Big ]\le d_p\eps^{p/2}.
		\end{aligned}
	\end{equation}
	
	Then by \eqref{prop1_6} we have for $p \in [1, \infty)$
	\begin{equation} \label{prop1_7a}
		\begin{aligned}
			&\PP(\{X_{t-\eps} \in B_{\delta - \delta_0/2}(\xi)\} \cap \{\tau_{t, \eps} < t\}) \\
			& \qquad \qquad \le \PP(\{ |X_t^{\tau_{t, \eps}} - X_{t-\eps}^{\tau_{t, \eps}}| \ge \delta_0 / 2\})\le \big( 2 / \delta_0 \big)^p \EE[|X_t^{\tau_{t, \eps}} - X_{t-\eps}^{\tau_{t, \eps}}|^p] \le d_p \big( 2 / \delta_0 \big)^p \eps^{p/2}.
		\end{aligned}
	\end{equation}
	Similar to the proof of Theorem 4.1 in~\cite{rom18} we consider stopping times
	\begin{equation*}
		\begin{aligned}
			\hat{\tau}_{t, \eps}^{(1)} &:= \inf\{s \in [t-\eps, t] \colon X_s \in [\xi-\delta + \frac{3\delta_0}{4}, \xi + \delta - \frac{3\delta_0}{4}]\}, \\
			\hat{\tau}_{t, \eps}^{(2)} &:= \inf\{s \in [t-\eps, t] \colon s \geq \hat{\tau}_{t, \eps}^{(1)} \,\text{and}\,  X_s \notin (\xi-\delta , \xi + \delta)\}, \\
			\hat{\tau}_{t, \eps}^{(3)} &:= \inf\{s \in [t-\eps, t] \colon  s \geq \hat{\tau}_{t, \eps}^{(1)} \,\text{and}\, X_s \in [\xi-\delta + \frac{7\delta_0}{8}, \xi + \delta - \frac{7\delta_0}{8}]\}
		\end{aligned}
	\end{equation*}
	to obtain similar to \eqref{prop1_7a} for every $p \in [1, \infty)$ a constant $\hat{d}_p = \hat{d}_p(\norm{\mu}_{L^\infty([\xi - \delta, \xi+\delta])}) \in (0, \infty)$ with
	\begin{equation}\label{prop1_7b}
		\begin{aligned}
			&\PP(\{X_t \in [\xi - \delta + \delta_0, \xi + \delta - \delta_0]\} \cap \{X_{t-\eps} \in \R \setminus B_{\delta - \delta_0/2}(\xi)\}) \\
			&\qquad \qquad \le \PP(\{X_t \in [\xi - \delta + \delta_0, \xi + \delta - \delta_0]\} \cap \{\hat{\tau}_{t, \eps}^{(1)} < t\}) \\
			&\qquad \qquad \le \PP(\{\hat{\tau}_{t, \eps}^{(1)} < t\} \cap \{\hat{\tau}_{t, \eps}^{(2)} < t\} ) + \PP(\{\hat{\tau}_{t, \eps}^{(1)} < t\} \cap \{\hat{\tau}_{t, \eps}^{(2)} \ge t\} \cap \{\hat{\tau}_{t, \eps}^{(3)} < t\}) \\
			&\qquad \qquad \le \PP(|X_{t \land \hat{\tau}_{t, \eps}^{(2)}} - X_{t \land \hat{\tau}_{t, \eps}^{(1)}}| \ge \delta_0 / 4) + \PP(|X_{t \land \hat{\tau}_{t, \eps}^{(2)} \land \hat{\tau}_{t, \eps}^{(3)}} - X_{t \land \hat{\tau}_{t, \eps}^{(1)}}| \ge \delta_0 / 8)\\
			&\qquad \qquad \le \hat{d}_p \big( 8 / \delta_0\big)^p  \eps^{p/2}.
		\end{aligned}
	\end{equation} 
	We work in the following on the set
	\[
	A_{t, \eps} := \{X_{t-\eps} \in B_{\delta - \delta_0/2}(\xi)\} \cap \{ \tau_{t, \eps} \ge t\}
	\]
	on which the solution stays within the interval where the coefficients behave nice. For bounded complex valued random variables $U$ with $|U| \le K$ where $K \in (0, \infty)$ we can restrict to the set $A_{t, \eps}$ using \eqref{prop1_7a}, \eqref{prop1_7b} and the fact that $\varphi\vert_{\R \setminus [\xi - \delta + \delta_0, \xi + \delta - \delta_0]} = 0$ via
	\begin{equation}\label{prop1_7c}
		\begin{aligned}
			|\EE[1_{A_{t, \eps}^\mathrm{C}}U \varphi(X_{t - \eps})]| &\le K \cdot \| \varphi \|_\infty \cdot \PP(\{X_{t-\eps} \in B_{\delta - \delta_0/2}(\xi)\}  \cap A_{t, \eps}^\mathrm{C}) \le K \cdot \| \varphi \|_\infty  \cdot d_p \big( 2 / \delta_0 \big)^p \eps^{p/2}, \\
			|\EE[1_{A_{t, \eps}^\mathrm{C}} U \varphi(X_{t})]| &\le K \cdot \| \varphi \|_\infty \cdot \PP(\{X_t \in [\xi - \delta + \delta_0, \xi + \delta - \delta_0]\} \cap A_{t, \eps}^\mathrm{C}) \\
			&\le  K \cdot \| \varphi \|_\infty \cdot \Big( \hat{d}_p \big( 8 / \delta_0\big)^p  \eps^{p/2} + d_p \big( 2 / \delta_0 \big)^p \eps^{p/2} \Big),
		\end{aligned}
	\end{equation}
	where $p \in [1, \infty)$.
	\\
	\textit{Bound for the Fourier transform:}
	Let $y \in \R$. Exploiting \eqref{prop1_7c} we obtain
	\begin{equation*}
		\begin{aligned}
			|\F[\varphi(x) \, \PP^{X_t}(dx)] (y)|  &= |\int_\R e^{\bi yx} \varphi(x) \, \PP^{X_t}(dx)| = |\EE[e^{\bi y X_t} \varphi(X_t)]| \le |\EE[1_{A_{t,\eps}} e^{\bi y X_t} \varphi(X_t)]| + c_1 \eps \\
			&  = |\EE[1_{A_{t,\eps}} e^{\bi y Z_{t, \eps}} \varphi(X_t)] + \EE[1_{A_{t,\eps}}(e^{\bi y X_t} - e^{\bi y Z_{t, \eps}})\varphi(X_t)] | + c_1 \eps \\
			&  \le |\EE[1_{A_{t,\eps}} e^{\bi y Z_{t, \eps}} \varphi(X_t)]| + |y|\EE[1_{A_{t,\eps}}|(X_t - Z_{t, \eps})\varphi(X_t)|] + c_1 \eps.
		\end{aligned}
	\end{equation*}
	Next, we approximate $\varphi(X_t)$ with a Milstein-type step to get
	\begin{equation*}
		\begin{aligned}
			|\F[\varphi(x) \, \PP^{X_t}(dx)] (y)| & \le |\EE[1_{A_{t,\eps}} e^{\bi y Z_{t, \eps}} (\varphi(X_{t-\eps}) + \varphi'(X_{t-\eps})(X_t - X_{t-\eps}))] \\
			&  \qquad \qquad + \EE[1_{A_{t,\eps}} e^{\bi y Z_{t, \eps}} (\varphi(X_t) - \varphi(X_{t-\eps}) - \varphi'(X_{t-\eps})(X_t - X_{t-\eps}))]| \\
			& \qquad \qquad + |y|\EE[1_{A_{t,\eps}} |(X_t - Z_{t, \eps})\varphi(X_t)|] + c_1 \eps.
		\end{aligned}
	\end{equation*}
	An application of the standard inequality
	\[
	|\varphi(z_1) - \varphi(z_2) - \varphi'(z_2)(z_1 - z_2)| = |\int_{z_2}^{z_1} \varphi'(v) - \varphi'(z_2) \, dv| \le [\varphi']_{C^1(\R)} |z_1 - z_2|^2, \qquad z_1, z_2 \in \R,
	\]
	together with \eqref{prop1_6} therefore shows that 
	\begin{equation}\label{prop1_3}
		\begin{aligned}
			&|\F[\varphi(x) \, \PP^{X_t}(dx)] (y)| \\
			& \qquad  \le |\EE[1_{A_{t,\eps}} e^{\bi y Z_{t, \eps}} (\varphi(X_{t-\eps}) + \varphi'(X_{t-\eps})(X_t - X_{t-\eps}))]| + c_2 \EE[1_{A_{t,\eps}}|X_t - X_{t-\eps}|^2]  \\
			& \qquad \qquad \qquad  + |y|\EE[1_{A_{t,\eps}}|(X_t - Z_{t, \eps})\varphi(X_t)|] + c_1 \eps \\
				& \qquad =   |\EE[1_{A_{t,\eps}} e^{\bi y Z_{t, \eps}} (\varphi(X_{t-\eps}) + \varphi'(X_{t-\eps})(X_t - Z_{t, \eps} + Z_{t, \eps} - X_{t-\eps}))]| \\
			& \qquad \qquad \qquad  + c_2 \EE[1_{A_{t,\eps}}|X^{\tau_{t, \eps}}_t - X^{\tau_{t, \eps}}_{t-\eps}|^2] + |y|\EE[ 1_{A_{t,\eps}}|(X_t - Z_{t, \eps})\varphi(X_t)|] + c_1 \eps\\
			& \qquad \le |\EE[1_{A_{t,\eps}} e^{\bi y Z_{t, \eps}} (\varphi(X_{t-\eps}) + \varphi'(X_{t-\eps})(Z_{t, \eps} - X_{t-\eps}))]| + c_3 \eps + c_4(1 + |y|) \EE[1_{A_{t,\eps}}|X_t - Z_{t, \eps}|].
		\end{aligned}
	\end{equation}
	Now we have because of $\varphi\vert_{\R \setminus (\xi - \delta + \delta_0, \xi + \delta - \delta_0)} = 0$ with \eqref{prop1_7a}
	\begin{equation}\label{prop1_4}
		\begin{aligned}
			&|\EE[1_{A_{t, \eps}^\mathrm{C}}e^{\bi y Z_{t, \eps}} \varphi'(X_{t-\eps})(Z_{t, \eps} - X_{t-\eps})]| \\
			& \qquad \le \EE[1_{A_{t, \eps}^\mathrm{C}}|\varphi'(X_{t-\eps})|\cdot \sup_{x \in B_\delta(\xi)} |\eps \mu(x) + \sigma(x)(W_t - W_{t-\eps})|] \\
			&  \qquad \le c_5 \sqrt{\PP(\{X_{t-\eps} \in B_{\delta - \delta_0/2}(\xi)\} \cap \{\tau_{t, \eps} < t\}) }  \cdot \sqrt{\EE[\sup_{x \in B_\delta(\xi)} | \eps \mu(x) + \sigma(x)(W_t - W_{t-\eps})|^2 ]} \\
			& \qquad \le c_6 \eps.
		\end{aligned}
	\end{equation}
	This together with \eqref{prop1_7c} and \eqref{prop1_3} yields
	\begin{equation*}
		\begin{aligned}
			&|\F[\varphi(x) \, \PP^{X_t}(dx)] (y)| \\
			& \qquad \le |\EE[ e^{\bi y Z_{t, \eps}} (\varphi(X_{t-\eps}) + \varphi'(X_{t-\eps})(Z_{t, \eps} - X_{t-\eps}))]| + c_7 \eps + c_4(1+|y|)\EE[1_{A_{t,\eps}}|X_t - Z_{t, \eps}|].
		\end{aligned}
	\end{equation*}
	Combining \eqref{prop1_1} and \eqref{prop1_2}, we thus get because of $\sigma\vert_{[\xi-\delta, \xi+\delta]} = 1$ and $\varphi\vert_{\R \setminus B_\delta(\xi)} = 0$
	\begin{equation*}
		|\F[\varphi(x) \, \PP^{X_t}(dx)] (y)| \le c_8(1 + \eps + \eps |y|) e^{-y^2\eps/2} + c_7 \eps + c_4(1+|y|)  \EE[1_{A_{t,\eps}} |X_t - Z_{t, \eps}|].
	\end{equation*}
		Observing that because of $\sigma\vert_{[\xi - \delta, \xi+\delta]} = 1$ it holds for $\PP$-almost all $\omega \in \{\tau_{t, \eps} \ge t\}$
	\begin{equation*}
		|X_t - Z_{t, \eps}|(\omega) = \big|\int_{t-\eps}^t \mu(X_s(\omega)) - \mu(X_{t-\eps}(\omega)) \, ds \big|,
	\end{equation*}
	the claim follows.
\end{proof}

The general statement of Theorem~\ref{thm:localDensity} can now be reduced to the special case from Proposition~\ref{prop:fourierBound}
which is done in the following proof.

\begin{proof}[Proof of Theorem~\ref{thm:localDensity}]
	Theorem~\ref{thm:localDensity} is basically a consequence of an application of a Lamperti-type transform and Proposition~\ref{prop:fourierBound}. We introduce at first similar to Corollary 2.3 in~\cite{Banos17} and Section~6.2 in \cite{Hayashi2014} the Lamperti-type transform by
	\begin{equation}\label{prop1_H}
		H\colon \R \rightarrow \R, \quad x \mapsto \int_{\xi - \delta}^x \frac{1}{\sigma^\ast(z)} \, dz.
	\end{equation}
	Since $\sigma$ is Lipschitz continuous on $[\xi - \delta, \xi+\delta]$, $\sigma^\ast$ is Lipschitz continuous. Using the assumption $\inf_{x \in B_\delta(\xi)} |\sigma(x)| > 0$ one thus obtains that the derivative $H' = \frac{1}{\sigma^\ast}$ of $H$ is Lipschitz continuous and that $H$ is a strictly monotonic and bi-Lipschitz continuous function. Hence, an application of the It\^{o} formula, see e.g.~\cite[Problem 3.7.3]{ks91}, yields that $Y := H(X)$ is a strong solution of the SDE
	\begin{equation*}
		\begin{aligned}
			dY_t &= \mu^H(Y_s) \, ds + \sigma^H(Y_s) \, dW_s,\\
			Y_0&= H(x_0),
		\end{aligned}
	\end{equation*}
	where
	\[
	\mu^H = \Big(\frac{\mu}{\sigma^\ast} - \frac{\sigma^2 \delta_{\sigma^\ast}}{2(\sigma^\ast)^2}\Big)\circ H^{-1}  \quad \text{and} \quad \sigma^H =  \frac{\sigma}{\sigma^\ast}\circ H^{-1}.
	\]
	We want to apply Proposition~\ref{prop:fourierBound} to $Y$ and $\varphi \circ H^{-1}$ and we therefore check the assumptions of the proposition. By the continuity and the strict monotonicity of $H$, there are a $\xi^H \in \R$ and a $\delta^H \in (0, \infty)$ such that $H(B_\delta(\xi)) = B_{\delta^H}(\xi^H)$ and it holds $\sigma^H(y)=\frac{\sigma}{\sigma^\ast} \circ H^{-1}(y)= 1$ for all \mbox{$y \in [\xi^H - \delta^H, \xi^H+\delta^H] = \overline{H(B_{\delta}(\xi))}$}. The boundedness of $\mu^{H}\vert_{[\xi^H - \delta^H, \xi^H+\delta^H]}$ is satisfied since it holds for $\mu\vert_{[\xi - \delta, \xi+\delta]}$ and $\sigma\vert_{[\xi - \delta, \xi+\delta]}$, since $ [\xi^H - \delta^H, \xi^H+\delta^H] = \overline{H(B_{\delta}(\xi))}$, $\sigma^\ast$ is Lipschitz continuous and $\sigma^\ast$ is bounded away from zero. Now the derivative $(\varphi \circ H^{-1})' = \frac{\varphi'\circ H^{-1}}{H' \circ H^{-1}} = (\varphi' \sigma^\ast) \circ H^{-1}$ is Lipschitz continuous since $\varphi', \sigma^\ast$ and $H^{-1}$ are Lipschitz continuous and since $\varphi$ has a bounded support. The continuity and strict monotonicity of $H$ also yield the existence of a \mbox{$\delta_0^H \in (0, \delta^H)$} such that $\supp(\varphi \circ H^{-1}) = H(\supp(\varphi)) \subset H(B_{\delta - \delta_0}(\xi)) \subset B_{\delta^H - \delta^H_0}(\xi^H)$. By Proposition~\ref{prop:fourierBound} there thus exists a constant \mbox{$c_1 = c_1(\norm{\mu}_{L^\infty([\xi - \delta, \xi+\delta])}, \norm{\sigma^\ast}_{C^1(\R)}, l_\sigma,\norm{\varphi}_{C^2(\R)}, \delta_0) \in (0, \infty)$} such that for all $t \in (0,1]$, $\varepsilon \in (0, t)$ and $y \in \R$ it holds
	\begin{equation*}
		\begin{aligned}
			&|\F[(\varphi \circ H^{-1})(x) \, \PP^{Y_t}(dx)] (y)|  \\
			&  \le c_1 \Big[(1 +  \eps |y|)e^{-(\eps y^2)/2} + \eps + (1+ |y|) \EE \bigl [1_{\{ \forall s \in [t-\eps, t]:|Y_s - \xi^H| \le \delta^H \}} \Big|\int_{t - \eps}^t \mu^H(Y_s) - \mu^H(Y_{t- \eps}) \, ds\Big|\bigr ] \Big].
		\end{aligned}
	\end{equation*}
	Using $B_{\delta^H}(\xi^H) = H(B_\delta(\xi))$ and plugging in the definitions of $\mu^H$ and $\sigma^\ast$ we thus get for all $t \in (0,1]$ and $\eps \in (0,t)$
	\begin{equation}\label{thm1_1}
		\begin{aligned}
			&|\F[(\varphi \circ H^{-1})(x) \, \PP^{Y_t}(dx)] (y)|  \\
			& \qquad \le c_1 \Big[(1 +  \eps |y|)e^{-(\eps y^2)/2} + \eps \\
			&\qquad \qquad  + (1+ |y|) \EE \bigl [1_{\{ \forall s \in [t-\eps, t]:|X_s - \xi| \le \delta \}} \Big|\int_{t - \eps}^t \Big(\frac{\mu}{\sigma^\ast} - \frac{\delta_{\sigma^\ast}}{2}\Big)(X_s) - \Big(\frac{\mu}{\sigma^\ast} - \frac{\delta_{\sigma^\ast}}{2}\Big)(X_{t- \eps}) \, ds\Big|\bigr ] \Big].
		\end{aligned}
	\end{equation}
	
	Let $t \in (0,1]$ and $\eps \in (0,t)$. In the following let $c_2,c_3,\dots\in (0,\infty)$ denote positive constants,
	which may depend on $\norm{\mu}_{L^\infty([\xi - \delta, \xi+\delta])}, \norm{\sigma^\ast}_{C^1(\R)}, l_\sigma,\norm{\varphi}_{C^2(\R)}, \delta_0$ and which neither depend on $t$ nor on $\eps$. Since $\frac{\mu}{\sigma^\ast} - \frac{\delta_{\sigma^\ast}}{2}$ is bounded on the set $[\xi - \delta, \xi + \delta]$ it holds with \eqref{thm1_1} for $y \in \R$
	\begin{equation*}
		|\F[(\varphi \circ H^{-1})(x) \, \PP^{Y_t}(dx)] (y)|  \le c_2 \big[ (1 + \eps |y|)e^{-(\eps y^2)/2} + \eps + (1 + |y|) \eps \big].
	\end{equation*}
	Choosing $\eps = \frac{\log^2(|y|)}{y^2}$ for $y \in \R \setminus \{0\}$ with $\frac{\log^2(|y|)}{y^2} < t$, the fact $|\F[(\varphi \circ H^{-1})(x) \, \PP^{Y_t}(dx)]| \le c_3$ thus yields $\F[(\varphi \circ H^{-1})(x) \, \PP^{Y_t}(dx)] \in L^2(\R)$ and hence the measure $(\varphi \circ H^{-1})(x) \, \PP^{Y_t}(dx)$ has a Lebesgue-density $p_t \in L^2(\R)$ with $\F p_t = \F[(\varphi \circ H^{-1})(x) \, \PP^{Y_t}(dx)]$ a.e., see e.g. Theorem 3.12 in \cite{Wolff2003}. With integration by substitution we now have 
	\[
	\varphi(x) \, \PP^{X_t}(dx) = p_t(H(x)) |H'(x)| \, dx =  p_t \bigl(\int_{\xi-\delta}^x \frac{1}{\sigma^\ast(z)} \, dz \bigr) \, \big/ \, |\sigma^\ast(x)| \, dx.
	\]
	Together with the fact that $\F p_t (y)= \F[(\varphi \circ H^{-1})(x) \, \PP^{Y_t}(dx)] (y) = \EE[e^{i y Y_t}(\varphi \circ H^{-1})(Y_t)]$ for $\lambda$-almost all $y \in \R$ and \eqref{thm1_1} this shows the claim.
\end{proof}

\section{Proof of Corollary~\ref{cor:HoelderReg}}\label{sect:proofCorollaries}

In this section we apply Corollary~\ref{cor1:fourierBound} to derive Corollary~\ref{cor:HoelderReg}. We will use the following classical result for the connection of the Fourier transform of a function and its Hölder regularity. For convenience of the reader we provide a short proof of the statement.

\begin{lemma}\label{lem:FourierHoelder}
	Let $f \in L^2(\R)$, $\gamma \in (0,1)$ and $c \in (0, \infty)$ such that for $\lambda$-almost all $y \in \R$ it holds
	\[
	|\F f (y)| \le \frac{c}{(1 + |y|)^{1 + \gamma}}.
	\]
	Let $\widetilde{f} = \frac{1}{2\pi} \int_\R e^{-\bi z(\cdot)}\, \F f(z)\, dz$. Then it holds $f = \widetilde{f}$ a.e. and there exists a constant $d_\gamma \in (0, \infty)$ which is independent of $f$ and $c$ such that
	\[
	\norm{\widetilde{f}}_{C^\gamma(\R)} \le d_\gamma c.
	\]
\end{lemma}
\begin{proof}
	The proof is based on the proofs of Corollary 1 in~\cite{Hayashi2012} and Lemma 16.3 in~\cite{TARTAR2007}. 
	Since $\F f \in L^1(\R)$, $\widetilde{f}$ is well-defined and the Fourier inversion theorem yields $f = \widetilde{f}$ a.e.. Let $x, y \in \R$ with $x \neq y$. Then we have
	\begin{equation*}
		\begin{aligned}
			&|\widetilde{f}(x) - \widetilde{f}(y)| = \frac{1}{2\pi} |\int_\R (e^{- \bi x z} - e^{- \bi y z}) \F f (z) \, dz| \le \frac{1}{2\pi} \int_\R |e^{-\bi (x-y) z} - 1| |\F f (z)| \, dz \\
			& \qquad \qquad \le \frac{c}{2 \pi} \int_\R \frac{2|\sin(\frac{(x-y)z}{2})|}{(1 + |z|)^{1+\gamma}} \, dz \le \frac{c}{2\pi} \int_\R \frac{2|\sin(\frac{(x-y)z}{2})|}{|z|^{1+\gamma}} \, dz.
		\end{aligned}
	\end{equation*}
	Substituting with $\R \rightarrow \R, \, z \mapsto \frac{z}{(x-y)}$ we thus obtain
	\begin{equation*}
		|\widetilde{f}(x) -\widetilde{f}(y)| \le \frac{c}{2\pi} |x-y|^\gamma \int_\R \frac{2|\sin(\frac{z}{2})|}{|z|^{1+\gamma}} \, dz.
	\end{equation*}
	Together with the basic inequality $\norm{\widetilde{f}}_\infty \le \frac{1}{2\pi} \norm{\F f}_{L^1(\R)}$ this finishes the proof.
\end{proof}

Now we are able to prove Corollary~\ref{cor:HoelderReg}.

\begin{proof}[Proof of Corollary~\ref{cor:HoelderReg}]
	The proof is divided into two steps. In the first step we derive the
	validity of \eqref{eqn_cor_localHoeld_1} and in the second step we show the continuity of $(0,1] \times \R \rightarrow \R, \, (t,x) \mapsto q_t(x)$.
	
	Let $t^\ast \in (0,1]$, $t \in [t^\ast, 1]$ and $\gamma \in (0, \alpha)$. In the following let $c_1,c_2,\dots\in (0,\infty)$ denote positive constants,
	which do not depend on $t$.

	\textit{Existence and H\"older continuity of $q_t$:}
	We set $\eps_y := \log^2(|y|) / y^2$ and
	\[
	\tau_{t, y} := \inf\{ s \in [t-\eps_y, t] \colon X_s \notin (\xi - \delta, \xi + \delta)\}
	\]  
	for $y \in \R \setminus \{0\}$ with $\frac{\log^2(|y|)}{|y|^2} < t$. Then with Corollary~\ref{cor1:fourierBound} there exists a function $p_t \in L^2(\R)$ which satisfies \eqref{thm1_density} such that for $\lambda$-almost all $y \in \R$ with $|y| > 1$ and $\frac{\log^2(|y|)}{|y|^2} < t$ it holds $| \F p_t(y)| \le \| \varphi \|_\infty$ as well as
	\begin{equation}\label{cor_jump_1}
		\begin{aligned}
			|\F p_t (y)| & \le c_1 \Big[ |y|^{-\log(|y|)/2} + \frac{\log^2(|y|)}{y^2} \\
			& \qquad +  |y|\EE \bigl [ \Big|\int_{t - \eps_y}^t \Big(\frac{\mu}{\sigma} - \frac{\delta_\sigma}{2}\Big)(X^{\tau_{t, y}}_{s}) - \Big(\frac{\mu}{\sigma} - \frac{\delta_\sigma}{2}\Big)(X^{\tau_{t, y}}_{t- \eps_y})\, ds\Big|\bigr ] \Big].
		\end{aligned}
	\end{equation} 
	Therefore, we obtain for $\lambda$-almost all $y \in \R$ with $|y| > 1$ and $\frac{\log^2(|y|)}{|y|^2} < t$  
	\begin{equation*}
		\begin{aligned}
			|\F p_t (y)| & \le c_2 \Big[ |y|^{-\log(|y|)/2} + \frac{\log^2(|y|)}{y^2} + |y|\eps_y \Big]
		\end{aligned}
	\end{equation*} 
	and hence there exists because of $| \F p_t| \le \| \varphi \|_\infty \text{ a.e.}$ for all $p \in (1, \infty)$ a constant $d^{(p)}_1 = d^{(p)}_1(t^\ast) \in (0, \infty)$ such that
	\begin{equation}\label{cor_jump_2}
		\big( \int_{\R} |\F p_t(y)|^p \, dy \big)^{1/p} < d^{(p)}_1.
	\end{equation}
	In the remaining part of this proof step we show that it holds for $\lambda$-almost all $y \in \R$
	\begin{equation}\label{cor_jump_3}
		|\F p_t(y)|  \le \frac{c_3}{(1 + |y|)^{1 + \gamma}}.
	\end{equation}
	Then with an application of Lemma~\ref{lem:FourierHoelder} we obtain for $\widetilde{p_t} = \frac{1}{2\pi} \int_\R e^{-\bi z(\cdot)}\, \F p_t(z)\, dz$
	\begin{equation*}
		p_t = \widetilde{p_t} \text{ a.e.} \qquad \text{and} \qquad \norm{\widetilde{p_t}}_{C^\gamma(\R)} \le c_4.
	\end{equation*} 
	We may assume 
	\begin{equation}\label{cor_jump_6}
		p_t = \widetilde{p_t} = \frac{1}{2\pi} \int_\R e^{-\bi z(\cdot)}\, \F p_t(z)\, dz. 
	\end{equation} 
 	According to \eqref{thm1_density} we take
	\[
	q_t(x) = p_t \bigl(\int_{\xi-\delta}^x \frac{1}{\sigma^\ast(z)} \, dz \bigr) \, \big/ \, |\sigma^\ast(x)|, \qquad x \in \R.
	\]
	Since $\supp(q_t) \subset \supp(\varphi)$, $\sigma^\ast$ is Lipschitz continuous and bounded away from zero, we then have $q_t \in L^1(\R)$ and $\sup_{t \in [t^\ast, 1]} \norm{q_t}_{C^\gamma(\R)} < \infty$.
	
	Subsequently, we show \eqref{cor_jump_3}. Using that $\frac{\mu}{\sigma} - \frac{\sigma'}{2}$ is $\alpha$-Hölder continuous on $[\xi - \delta, \xi)$ and on $(\xi, \xi + \delta]$, respectively, it holds
	\begin{equation*}
		\Big|\Big(\frac{\mu}{\sigma} - \frac{\sigma'}{2}\Big)(x) - \Big(\frac{\mu}{\sigma} - \frac{\sigma'}{2}\Big)(z)\Big| \le c_5 (|x-z|^\alpha + 1_{\{(x-\xi)(z-\xi) \le 0\}}), \quad x,z \in [\xi-\delta, \xi+\delta],
	\end{equation*}
	c.f. the proof of Lemma 1 in~\cite{MGY21}. Consequently, we obtain with \eqref{prop1_6} and \eqref{cor_jump_1} for $\lambda$-almost all $y \in \R$ with $|y| > 1$ and $\frac{\log^2(|y|)}{|y|^2} < t$ 
	\begin{equation}\label{cor_jump_4}
		\begin{aligned}
			&|\F p_t (y)| \\
			& \qquad \le c_6 \Big[ |y|^{-\log(|y|)/2} + \frac{\log^2(|y|)}{y^2} +  |y|\EE \bigl [ \int_{t - \eps_y}^t |X^{\tau_{t, y}}_{s} - X^{\tau_{t, y}}_{t- \eps_y}|^\alpha + 1_{\{(X^{\tau_{t, y}}_{s} - \xi)(X^{\tau_{t, y}}_{t - \eps_y} - \xi)\le 0\}}\, ds\bigr ] \Big] \\
			&\qquad \le c_7 \Big[ |y|^{-\log(|y|)/2} + \frac{\log^2(|y|)}{y^2} +  |y|(\eps_y)^{1 + \alpha/2} + |y| \int_{t - \eps_y}^t \PP((X^{\tau_{t, y}}_{s} - \xi)(X^{\tau_{t, y}}_{t - \eps_y} - \xi)\le 0)\, ds \Big].
		\end{aligned}
	\end{equation} 
	Let $\beta \in (0, 1/2)$ and $p \in (1, \infty)$. Now there exists with \eqref{prop1_6} a constant $d^{(p)}_2 \in (0,\infty)$ such that for all $y \in \R$ with $|y| > 1$ and $\frac{\log^2(|y|)}{|y|^2} \le t/2$ and all $s \in (t - \eps_y, t)$ it holds
	\begin{equation}\label{cor_jump_5}
		\begin{aligned}
			&\PP((X^{\tau_{t, y}}_s - \xi)(X^{\tau_{t, y}}_{t - \eps_y} - \xi)\le 0)\\
			 &\qquad \qquad  \le \PP(|X^{\tau_{t, y}}_{t - \eps_y} - \xi| \le (\eps_y)^{1/2 - \beta}) + \PP(|X^{\tau_{t, y}}_s - X^{\tau_{t, y}}_{t-\eps_y}| \ge (\eps_y)^{1/2 - \beta}) \\
			& \qquad \qquad \le \PP(|X^{\tau_{t, y}}_{t - \eps_y} - \xi| \le (\eps_y)^{1/2 - \beta}) + (\eps_y)^{p(\beta- 1/2)} \EE[|X^{\tau_{t, y}}_s - X^{\tau_{t, y}}_{t-\eps_y}|^p] \\
			&\qquad \qquad \le \PP(|X^{\tau_{t, y}}_{t - \eps_y} - \xi| \le (\eps_y)^{1/2 - \beta}) + d^{(p)}_2 (\eps_y)^{p\beta}.
		\end{aligned}
	\end{equation}
	Let $y \in \R$ with $|y| > 1, \frac{\log^2(|y|)}{|y|^2} \le t/2$ and $(\eps_y)^{1/2 - \beta} \le \delta/2$. Let $\tphi \in C^\infty(\R)$ with $\supp(\tphi) \subset B_{(3/4)\delta}(\xi)$ and $\tphi(x) = 1$ for $x \in B_{\delta/2}(\xi)$. Let $p^{\tphi}_t$ denote the function $p_t$ which one obtains with Theorem~\ref{thm:localDensity} with $\tphi$ instead of $\varphi$. We get for $H$ as in \eqref{prop1_H} with integration by substitution  because of $t - \eps_y \ge t/2 \ge t^\ast/2$, $\tau_{t, y} \ge t - \eps_y$, $(\eps_y)^{1/2 - \beta} \le \delta/2$ and $\tphi \vert_{B_{\delta/2}(\xi)} = 1$
	\begin{equation*}
		\begin{aligned}
			&\PP(|X^{\tau_{t, y}}_{t - \eps_y} - \xi| \le (\eps_y)^{1/2 - \beta})\\
			& \qquad \qquad = \PP(|X_{t - \eps_y} - \xi| \le (\eps_y)^{1/2 - \beta}) =  \int_{\xi - (\eps_y)^{1/2 - \beta}}^{\xi + (\eps_y)^{1/2 - \beta}} \tphi(x) \, \PP^{X_{t-\eps_y}}(dx) \\
			&\qquad \qquad  = \int_{\xi - (\eps_y)^{1/2 - \beta}}^{\xi + (\eps_y)^{1/2 - \beta}} p_t^{\tphi}(H(x))|H'(x)| \, dx  = \int_{H(B_{(\eps_y)^{1/2 - \beta}}(\xi))} p_t^{\tphi}(x) \, dx.
		\end{aligned}
	\end{equation*}
	Since $H$ is Lipschitz continuous and strictly monotonic as seen in the proof of Theorem~\ref{thm:localDensity}, we therefore obtain with \eqref{cor_jump_2} for $\hat p \in (1,2]$ and $\hat q \in (1, \infty)$ with $\frac{1}{\hat p} + \frac{1}{\hat{q}} = 1$ using the Fourier inversion theorem and the Hausdorff-Young inequality for a ${\widetilde d}^{(\hat p)}_1 = {\widetilde d}^{(\hat p)}_1(t^\ast) \in (0, \infty)$
	\begin{equation*}
		\begin{aligned}
			& \PP(|X_{t - \eps_y} - \xi| \le (\eps_y)^{1/2 - \beta}) \\
			& \qquad \qquad\le \big( \lambda(H(B_{(\eps_y)^{1/2 - \beta}}(\xi))) \big)^{1/\hat p} \cdot \big( \int_\R |p_t^{\tphi}(x)|^{\hat q} \, dx \big)^{1/\hat q} \\
			& \qquad \qquad \le (c_9)^{1/\hat p} (\eps_y)^{(1/2 - \beta)/\hat p} \cdot \big( \int_\R |\frac{1}{2\pi} \F [\F p_t^{\tphi}](x)|^{\hat q} \, dx \big)^{1/\hat q} \\
			& \qquad \qquad \le (c_9)^{1/\hat p} (\eps_y)^{(1/2 - \beta)/\hat p} \cdot (2\pi)^{-1+1/\hat q} \big( \int_\R |\F p_t^{\tphi}(x)|^{\hat p} \, dx \big)^{1/\hat p} \\
			& \qquad \qquad \le  {\widetilde d}^{(\hat p)}_1 (\eps_y)^{(1/2 - \beta)/\hat p}.
		\end{aligned}
	\end{equation*}
	Choosing $\beta \in (0, 1/2), \hat p \in (1, 2]$ small and $p$ in \eqref{cor_jump_5} large enough shows with \eqref{cor_jump_4} and the fact that $|\F p_t| \le \| \varphi \|_\infty$ a.e. the desired bound in \eqref{cor_jump_3}.
	
	\textit{Continuity of the map $(0,1] \times \R \rightarrow \R, \, (t,x) \mapsto q_t(x)$:}
	In the following we derive the continuity of the function $(0,1] \times \R \rightarrow \R, \, (t,x) \mapsto q_t(x)$. Note that by the choice of $q_t$ it suffices to show that $(0,1] \times \R \rightarrow \R, \, (t,x) \mapsto p_t(x)$ is continuous. Therefore, let $s, t \in (0,1]$, $x,y \in \R$, let $H\colon \R \rightarrow \R$ be as in \eqref{prop1_H} and set $Y = H(X)$. As can be seen in the proof of Theorem~\ref{thm:localDensity} we have
	\begin{equation} \label{eqn_HoelderReg_2}
		(\varphi \circ H^{-1})(x) \, \PP^{Y_t}(dx) = p_t(x) \, dx.
	\end{equation}
	Now assumption \eqref{cor_jump_6} yields
	\begin{equation*}
		\begin{aligned}
			|p_s(x) - p_t(x)| = \frac{1}{2\pi} |\int_\R (\F p_s (z) - \F p_t (z)) e^{- \bi x z}  \, dz| \le \frac{1}{2\pi}  \norm{\F p_s- \F p_t}_{L^1(\R)}.
		\end{aligned}
	\end{equation*}
	Let $(s_n)_{n \in \N} \subset (0,1]$ be a sequence with $\lim_{n \rightarrow \infty} s_n = t$. By~\eqref{eqn_HoelderReg_2} it holds for $\lambda$-almost all $u \in \R$
	\begin{equation*}
		\begin{aligned}
			&\F p_{s_n} (u) = \F[(\varphi \circ H^{-1})(x) \, \PP^{Y_{s_n}}(dx)] (u) = \EE[e^{\bi u Y_{s_n}} (\varphi \circ H^{-1})(Y_{s_n})] 
			\\& \qquad \qquad \rightarrow \EE[e^{\bi u Y_t} (\varphi \circ H^{-1})(Y_t)] = \F p_t (u)
		\end{aligned}
	\end{equation*}
	as $n \rightarrow \infty$. Assume that $t^\ast \in (0,1]$ is small enough and $s_n \ge t^\ast$ for all $n \in \N$. Now with \eqref{cor_jump_3} and \eqref{cor_jump_6} it holds for $\lambda$-almost all $y \in \R$ 
	\[
	|\F p_v(y)| \le \frac{c_3}{(1+ |y|)^{1+\gamma}}, \qquad v \in \{t\} \cup \{s_n \colon n \in \N\}.
	\] 
	Therefore, an application of the dominated convergence theorem yields the continuity of the map $(0,1] \rightarrow \R, \, t \mapsto p_t(x)$. Combining this with
	\begin{equation*}
		\begin{aligned}
			|p_s(y) - p_t(x)|  \le |p_s(y) - p_s(x)| + |p_s(x) - p_t(x)| \le \norm{p_s}_{C^{\gamma}(\R)}|y - x|^{\gamma} + |p_s(x) - p_t(x)|
		\end{aligned}
	\end{equation*}
	and $\sup_{s \in [t/2, 1]} \norm{p_s}_{C^\gamma(\R)} < \infty$ yields the claim.
\end{proof}

\section*{Acknowledgement}
I want to express my gratitude to Thomas M\"uller-Gronbach and Larisa Yaroslavtseva for their encouragement and useful critiques of this article.

\bibliographystyle{acm}
\bibliography{bibfile}

\end{document}